\documentclass[a4paper,11pt]{amsart}
\usepackage{amssymb,amscd,amsmath}

\theoremstyle{plain}
\newtheorem{thm}{Theorem}[section]
\newtheorem{prop}[thm]{Proposition}
\newtheorem{lem}[thm]{Lemma}

\theoremstyle{definition}

\theoremstyle{remark}
\newtheorem{rem}{Remark}[section]

\newcommand{\forme}[1]{}

\title{Terwilliger algebras of wreath products by 3-equivalenced schemes}

\author[K.~Kim]{Kijung Kim}
\address{Department of Mathematics, Pusan National University, Busan 609-735, Republic of Korea}
\email{knukkj@pusan.ac.kr}

\date{\today}

\begin{document}
\maketitle

\begin{abstract}
Recently G. Bhattacharyya, S.Y. Song and R. Tanaka began to study Terwilliger algebras of wreath products of one-class association schemes.
K. Kim determined the structure of Terwilliger algebras of wreath products by one-class association schemes or quasi-thin schemes.
In this paper, we study Terwilliger algebras of wreath products by $3$-equivalenced schemes.
\end{abstract}

\footnotetext{This research was supported by Basic Science Research Program through the National Research Foundation of Korea(NRF) funded by the Ministry of Education (2013R1A1A2005349).}

\section{Introduction}\label{sec:intro}

The Terwilliger algebra of an association scheme is an algebraic tool for the study of association schemes.
In general, Terwilliger algebras are non-commutative, finite dimensional, and semisimple $\mathbb{C}$-algebras.
Also they are more combinatorial and complicated than Bose-Mesner algebras.
Recently G. Bhattacharyya, S.Y. Song and R. Tanaka began to study Terwilliger algebras of wreath products of one-class association schemes in \cite{song}.
Generalizations and related works were developed by the authors of \cite{rie, hkm, kim, rieZies, Xu}.
In particular, the structure of Terwilliger algebras of wreath products by one-class association schemes or quasi-thin schemes was determined in \cite{kim}.
In this way, the attempt of generalization seems to be far from a success.
However, we propose studying Terwilliger algebras of wreath products by $k$-equivalenced schemes.

For a given scheme, its Terwilliger algebra is contained in one point extension of the scheme.
For a positive integer $k$,
an association scheme with more then one point is called \textit{k-equivalenced} if each non-diaginal basic relation has valency $k$.
A coherent configuration $(X,S)$ is called \textit{semiregular} if $|xs| \leq 1$ for $x \in X, s \in S$, where $xs:= \{ y \in X \mid (x, y) \in s \}$.
In \cite[Lemma 5.13]{mpp}, the following lemma was proved.
\begin{lem}\label{lem:impri}
Let $(X,S)$ be an imprimitive $k$-equivalenced scheme.
Then $(X \setminus \{ x \}, (S_x)_{X \setminus \{ x \}})$ is semiregular, where $S_x$ is the basic relation set of the one point extension of $(X,S)$.
\end{lem}

If the following problems for a given $k \geq 4$ are solved,
then it seems to be close to determine the structure of Terwilliger algebras of wreath products by $k$-equivalenced schemes.
\begin{enumerate}
\item When Terwilliger algebras of $k$-equivalenced schemes coincide with their one point extensions ?
\item When the imprimitivity condition of Lemma \ref{lem:impri} can be removed ?
\end{enumerate}

In this paper, we determine the structure of Terwilliger algebras of wreath products by $3$-equivalenced schemes.
Furthermore, it supports the importance of the above two problems and provides another motivation to study $k$-equivalenced schemes.

This paper is organized as follows.
In Section \ref{sec:pre}, we review notations and basic results on coherent configurations and Terwilliger algebras
as well as recent results on $3$-equivalenced schemes.
In Section \ref{sec:main}, based on the fact that the adjacency algebras of the one point extensions of $3$-equivalenced schemes coincide with their Terwilliger algebras, we determine all central primitive idempotents of Terwilliger algebras of wreath products by $3$-equivalenced schemes.
In Section \ref{sec:mainresult}, we state our main theorem.

\section{Preliminaries}\label{sec:pre}
In this section, we use the notations and terminology given in \cite{ep, hkm}.

\subsection{Coherent configurations and adjacency algebras}
Let $X$ be a finite set and $S$ a partition of $X \times X$.
Put by $S^\cup$ the set of all unions of the elements of $S$.
A pair $\mathcal{C}=(X,S)$ is called a \textit{coherent configuration} on $X$ if the following conditions hold:
\begin{enumerate}
\item $1_X = \{(x,x)\mid x\in X \} \in S^\cup$;
\item For $s \in S$, $s^{\ast}=\{(y,x)\mid (x,y)\in s\}\in S$;
\item For all $s,t,u\in S$ and all $x, y \in X$,
$$p_{st}^{u}=| \{z\in X\mid (x,z)\in s, \ (z,y)\in t\} |$$
is constant whenever $(x,y)\in u$.
\end{enumerate}
The elements of $X$, $S$ and $S^\cup$ are called the \textit{points}, the \textit{basic relations} and the \textit{relations}, respectively.
The numbers $|X|$ and $|S|$ are called the \textit{degree} and \textit{rank}, respectively.
Any set $\Delta \subseteq X$ for which $1_\Delta:=\{ (x, x) \mid x \in \Delta \} \in S$ is called the \textit{fiber}. The set of all fibers is denoted by $Fib(\mathcal{C})$.
For $s, t \in S$, the product $st$ is defined by $\{ u \in S \mid p_{st}^u \neq 0 \}$.
The coherent configuration $\mathcal{C}$ is called \textit{homogeneous} or a \textit{scheme} if $1_X \in S$.
If $Y$ is a union of fibers, then the \textit{restriction} of $\mathcal{C}$ to $Y$ is defined to be a coherent configuration
\[\mathcal{C}_Y =(Y,S_Y),\]
where $S_Y$ is the set of all non-empty relations $s \cap (Y\times Y)$ with $s \in S$.

Two coherent configurations $(X,S)$ and $(Y,T)$ are called \textit{isomorphic} if there exists a bijection $\tau : X \cup S \rightarrow Y \cup T$
such that $X^\tau = Y$, $S^\tau = T$ and $(ws)^\tau = (w^\tau)s^\tau$ for all $s \in S$ and $w \in X$ with $ws \neq \emptyset$.
We denote it by $(X,S) \simeq (Y,T)$.

For $s \in S$, let $\sigma_{s}$ denote the matrix in $Mat_{X}(\mathbb{C})$ that has entries
\[(\sigma_{s})_{xy}  = \left\{
                      \begin{array}{ll}
                      1 & \hbox{if $(x,y) \in s$;} \\
                      0 & \hbox{otherwise.}
                      \end{array}
                     \right.\]
We call $\sigma_{s}$ the \textit{adjacency matrix} of $s \in S$.
Then $\bigoplus_{s\in S}\mathbb{C}\sigma_{s}$ becomes a subalgebra of
$Mat_X(\mathbb{C})$. We call $\bigoplus_{s\in S}\mathbb{C}\sigma_{s}$ the
\textit{adjacency algebra} of $S$, and denote it by $\mathcal{A}(S)$.

Let $B$ be the set of primitive idempotents of $\mathcal{A}(S)$ with respect to the Hadamard multiplication.
Then $B$ is a linear basis of $\mathcal{A}$ consisting of $\{0, 1 \}$-matrices such that
$$\sum_{s \in B} \sigma_{s} = J_X  ~\text{and}~  \sigma_{s} \in B \Leftrightarrow \sigma_{s}^t \in B.$$

\begin{rem} \label{rem:1}
There are bijections between the sets of coherent configurations and adjacency algebras as follows:
$$S \mapsto \mathcal{A}(S)  ~\text{and}~  \mathcal{A}\mapsto \mathcal{C}(\mathcal{A}),$$
where $\mathcal{C}(\mathcal{A}) = (X, S')$ with $S'=\{ s \in X\times X \mid \sigma_{s} \in B\}$.
\end{rem}

In a scheme $(X, S)$, for each $s \in S$, $n_s := p_{ss^\ast}^{1_X}$ is called \textit{valency} of $s$.

\begin{lem}[\cite{zies}]\label{lem:constant}
Let $(X, S)$ be a scheme.
For $u, v, w \in S$, we have the following:
\begin{enumerate}
\item $p_{uw}^v n_v = p_{u^\ast v}^w n_w = p_{v w^\ast}^u n_u$;
\item $n_u n_v = \sum_{s \in S} p_{uv}^s n_s$.
\end{enumerate}
\end{lem}

Let $(X, S)$ be a scheme and $u, v \in S$.
By Lemma \ref{lem:constant}(i), we have $p_{uw}^v =0$ if and only if $p_{u^\ast v}^w = 0$.
Thus, given $x \in X$, the cardinality of the set
\[ \{ w \cap (xu \times xv) \mid w \in S, p_{uw}^v \neq 0 \} \]
equals the number $|u^\ast v|$.

\subsection{Terwilliger algebras and one point extensions}
Let $(X,S)$ be a scheme. For $U\subseteq X$, we denote by $\varepsilon_{U}$
the diagonal matrix in $Mat_{X}(\mathbb{C})$ with entries
\[(\varepsilon_{U})_{xx}  = \left\{
                      \begin{array}{ll}
                      1 & \hbox{if $x\in U$;} \\
                      0 & \hbox{otherwise.}
                      \end{array}
                     \right.\]
Note that
\[J_{U,V}:= \varepsilon_{U} J_X \varepsilon_{V} ~\text{and}~ J_U := J_{U,U}\]
for $U, V \subseteq X$.

The  \textit{Terwilliger algebra} of $(X,S)$ with respect to $x_{0}\in X$ is defined as a subalgebra of $Mat_X(\mathbb{C})$ generated by
\[ \{\sigma_{s} \mid s\in S\} \cup \{\varepsilon_{x_{0}s} \mid s \in S\} \]
(see \cite{terwilliger}).
The Terwilliger algebra will be denoted by $\mathcal{T}(X,S,x_{0})$ or $\mathcal{T}(S)$.
Since $\mathcal{A}(S)$ and $\mathcal{T}(S)$ are closed under transposed conjugate,
they are semisimple $\mathbb{C}$-algebras.
The set of irreducible characters of $\mathcal{T}(S)$ and $\mathcal{A}(S)$ will be denoted by
$\mathrm{Irr}(\mathcal{T}(S))$ and $\mathrm{Irr}(\mathcal{A}(S))$,
respectively. The \textit{trivial character} $1_{\mathcal{A}(S)}$ of $\mathcal{A}(S)$
is a map $\sigma_{s}\mapsto p_{ss^\ast}^{1_X}$ and the corresponding central
primitive idempotent is $|X|^{-1}J_{X}$. The \textit{trivial character} $1_{\mathcal{T}(S)}$ of $\mathcal{T}(S)$
corresponds to the central primitive idempotent \[\sum_{s\in S}n_{s}^{-1}\varepsilon_{x_{0}s}J_{X}\varepsilon_{x_{0}s}\]
of $\mathcal{T}(S)$. For $\chi\in \mathrm{Irr}(\mathcal{A}(S))$ or $\mathrm{Irr}(\mathcal{T}(S))$, $e_{\chi}$ will be the corresponding
central primitive idempotent of $\mathcal{A}(S)$ or $\mathcal{T}(S)$.
For convenience, we denote $\mathrm{Irr}(\mathcal{A}(S)) \setminus \{1_{\mathcal{A}(S)}\}$ and
$\mathrm{Irr}(\mathcal{T}(S)) \setminus \{1_{\mathcal{T}(S)}\}$ by $\mathrm{Irr}(\mathcal{T}(S))^\times$ and $\mathrm{Irr}(\mathcal{A}(S))^\times$, respectively.

Let $\mathcal{C}=(X,S)$ be a coherent configuration and $x \in X$. Denote by $S_x$ the set of basic relations of the smallest coherent configuration on $X$ such that
$$1_x \in S_x ~\text{and}~ S \subset S_x^\cup.$$
Then the coherent configuration $\mathcal{C}_x=(X,S_x)$ is called a \textit{one point extension} of $\mathcal{C}$.
It is easy to see that given $s, t, u \in S$, the set $xs$ is a union of fibers of $\mathcal{C}_x$
and the relation $u_{xs,xt}$ is a union of basic relations of $\mathcal{C}_x$.

\begin{rem} \label{rem:2}
The adjacency algebra of the one point extension $\mathcal{C}_x$ of a scheme $\mathcal{C}$ is related to $\mathcal{T}(X,S,x)$.
In fact, $\mathcal{T}(X,S,x) \subseteq \mathcal{A}(S_x)$.
\end{rem}

\subsection{Wreath products}
Let $(X,S)$ and $(Y,T)$ be schemes.
For $s\in S$, set
\[ \tilde{s}=\{((x,y),(x^{\prime},y)) \mid (x,x^{\prime})\in s, \ y \in Y\}. \]

For $t\in T$, set
\[ \bar{t}=\{((x,y),(x^{\prime},y^{\prime})) \mid x,x^{\prime}\in X, \ (y,y^{\prime})\in t\}.\]

Also set
\[ S\wr T=\{\tilde{s}\mid s\in S\}\cup\{\bar{t}\mid t\in T\setminus\{1_Y\}\}. \]

Then $(X\times Y, S\wr T)$ is a scheme called the \textit{wreath product} of $(X,S)$ by $(Y,T)$.

\vskip15pt
For the adjacency matrices, we have $\sigma_{\tilde{s}}=\sigma_{s}\otimes I_{Y}, \sigma_{\bar{t}}=J_{X}\otimes \sigma_{t}$.
Note that
\begin{eqnarray*}
(x_{0},y_{0})\tilde{s} &=(x_{0}s,y_{0}) &=\{(x,y_{0}) \mid x\in x_{0}s\}, \\
(x_{0},y_{0})\bar{t}   &=(X,y_{0}t)     &=\{(x,y) \mid x\in X, \ y\in y_{0}t\}
\end{eqnarray*}
and
\begin{eqnarray*}
\varepsilon_{(x_{0},y_{0})\tilde{s}}
& =\varepsilon_{x_{0}s}\otimes \varepsilon_{y_{0}1_Y},  \\
\varepsilon_{(x_{0},y_{0})\bar{t}}
& =\sum_{s\in S}\varepsilon_{x_{0}s}\otimes \varepsilon_{y_{0}t} &=I_{X}\otimes \varepsilon_{y_{0}t}.
\end{eqnarray*}

\subsection{3-equivalenced schemes}
A scheme $\mathcal{C}=(Y,T)$ is called \textit{3-equivalenced} if $T = \{ 1_Y \} \cup T_3$, where $T_3$ is the set of basic relations with valency $3$.

\begin{lem}[\cite{hkp}, Lemma 3.1]\label{lem:inter1}
Let $\mathcal{C}=(Y,T)$ be a 3-equivalenced scheme.
For all $u, v \in T^\times := T \setminus \{ 1_Y \}$ with $u \neq v^\ast$, one of the following holds:
\begin{enumerate}
\item $\sigma_u \sigma_v = \sigma_{w_1} + \sigma_{w_2} + \sigma_{w_3}$ for some distinct $w_1, w_2, w_3 \in T^\times$;
\item $\sigma_u \sigma_v = \sigma_{w_1} + 2\sigma_{w_2}$ for some distinct $w_1, w_2 \in T^\times$.
\end{enumerate}
\end{lem}

\begin{lem}[\cite{hkp}, Lemma 3.2]\label{lem:inter2}
Let $\mathcal{C}=(Y,T)$ be a 3-equivalenced scheme.
For all $u \in T^\times$, one of the following holds:
\begin{enumerate}
\item $\sigma_u \sigma_{u^\ast} = 3\sigma_{1_Y} + \sigma_{w} + \sigma_{w^\ast}$ for some $w \in T$ with $w \neq w^\ast$;
\item $\sigma_u \sigma_{u^\ast} = 3\sigma_{1_Y} + 2\sigma_{u}$ and $u = u^\ast$.
\end{enumerate}
\end{lem}

\begin{lem}[\cite{hkp}, Lemma 3.14]\label{lem:semiregular}
Let $\mathcal{C}=(Y,T)$ be a 3-equivalenced scheme with $|T| > 2$ and $y_{0} \in Y$.
Then the coherent configuration $(Y \setminus \{ y_{0} \}, (T_{y_{0}})_{Y \setminus \{ y_{0} \}})$ is semiregular.
\end{lem}

\begin{thm}\label{thm:coherent-terwilliger}
Let $\mathcal{C}=(Y,T)$ be a 3-equivalenced scheme and $y_{0} \in Y$.
Then $\mathcal{A}(T_{y_{0}}) = \mathcal{T}(Y,T,y_{0})$.
\end{thm}
\begin{proof}
It follows from \cite[Theorem 1.1]{hkp} and \cite[Section 8.5]{mp1}
that $\mathcal{A}(T_{y_{0}}) = \mathcal{T}(Y,T,y_{0})$.
\end{proof}

\section{Wreath products by 3-equivalenced schemes}\label{sec:main}
Let $(X,S)$ and $(Y,T)$ be schemes. Fix $x_{0}\in X$ and
$y_{0}\in Y$, and consider $(X\times Y, S\wr T)$ and $\mathcal{T}(X\times Y, S \wr T, (x_0, y_0))$. In
the rest of this section, we assume that $(Y,T)$ is a 3-equivalenced scheme with $|T| > 2$.

\subsection{The restriction of $\mathcal{T}(S \wr T)$ to $X \times (Y\setminus \{y_0\})$}
The contents of this subsection are given in \cite{kim}.

The Terwilliger algebra $\mathcal{T}(S \wr T)$ is generated by
\[\{ J_{X}\otimes \sigma_{t}, I_{X}\otimes \varepsilon_{y_{0}t} \mid t \in T\setminus \{1_Y \} \}
\cup \{ \sigma_{s}\otimes I_{Y}, \varepsilon_{x_{0}s}\otimes \varepsilon_{\{y_{0}\}} \mid s \in S \}.\]

Since $\sum_{s \in S} \varepsilon_{x_{0}s}\otimes \varepsilon_{\{y_{0}\}} = I_{X}\otimes \varepsilon_{\{y_{0}\}}$ and
$\sum_{s \in S} \sigma_{s}\otimes I_{Y} = J_{X}\otimes I_{Y}$, we consider a subalgebra $\mathcal{U}$ generated by
\[ \{ J_{X}\otimes \sigma_{t}, I_{X}\otimes \varepsilon_{y_{0}t} \mid t \in T \}. \]

It is easy to see that $\mathcal{U}$ is generated by
\[\{ |X|^{-1}J_{X}\otimes \varepsilon_{y_{0}t_1} \sigma_{t} \varepsilon_{y_{0}t_2} \mid t_1, t_2 \in T \}\]
and isomorphic to $\mathcal{T}(T)$.
So by Theorem \ref{thm:coherent-terwilliger}, a basis $B(\mathcal{U})$ of $\mathcal{U}$ can be determined by $\mathcal{C}(\mathcal{T}(T))$, i.e. $B(\mathcal{U}) = \{ J_{X}\otimes \sigma_{c} \mid c \in \mathcal{R}\}$,
where $\mathcal{R}$ is the set of basic relations of $\mathcal{C}(\mathcal{T}(T))$.

\vskip15pt
We consider \[ (\varepsilon_{X}\otimes \varepsilon_{Y\setminus  \{y_{0}\}}) \mathcal{T}(S \wr T) (\varepsilon_{X}\otimes \varepsilon_{Y\setminus  \{y_{0}\}}).\]

Since $\mathcal{T}(S \wr T)$ is generated by
$B(\mathcal{U}) \cup \{ \sigma_{s}\otimes I_{Y}, \varepsilon_{x_{0}s}\otimes \varepsilon_{\{y_{0}\}} \mid s \in S \}$,
$(\varepsilon_{X}\otimes \varepsilon_{Y\setminus  \{y_{0}\}}) \mathcal{T}(S \wr T) (\varepsilon_{X}\otimes \varepsilon_{Y\setminus  \{y_{0}\}})$
is generated by
\[ \{ J_{X}\otimes \sigma_{c} \mid c \in \mathcal{R}_{Y\setminus \{y_0\}} \} \cup \{ \sigma_{s}\otimes I_{Y\setminus \{y_0\}} \mid s \in S\}. \]

Thus, we can determine a basis of $(\varepsilon_{X}\otimes \varepsilon_{Y\setminus  \{y_{0}\}}) \mathcal{T}(S \wr T) (\varepsilon_{X}\otimes \varepsilon_{Y\setminus  \{y_{0}\}})$ with respect to the set of basic relations of $\mathcal{C}(\mathcal{T}(T))_{Y\setminus \{y_0\}}$.

\subsection{A basis of $\mathcal{A}(T_{y_{0}}) = \mathcal{T}(Y, T, y_0)$}

By Theorem \ref{thm:coherent-terwilliger}, we have $\mathcal{A}(\mathcal{C}_{y_{0}}) = \mathcal{T}(Y, T, y_0)$.

Let $W$ be $\{ r \in T_{y_{0}} \mid |y_0r|=3\}$.
In this paper, we shall say that the point set of $T_{y_{0}}$ is \textit{well-ordering} if
for all $r, r' \in W$ and $t \in T_{y_{0}}$,
the adjacency matrix of $t \cap (y_{0}r \times y_{0}r') \neq \emptyset$ is one of the followings:

\begin{equation}\label{A}
\begin{bmatrix}
1 & 0 & 0 \\
0 & 1 & 0 \\
0 & 0 & 1
\end{bmatrix}
;
\quad
\begin{bmatrix}
0 & 1 & 0 \\
0 & 0 & 1 \\
1 & 0 & 0
\end{bmatrix}
;
\quad
\begin{bmatrix}
0 & 0 & 1 \\
1 & 0 & 0 \\
0 & 1 & 0
\end{bmatrix}
,
\end{equation}
where the adjacency matrix of $t \cap (y_{0}r \times y_{0}r') \neq \emptyset$ means a $3 \times 3$ matrix
whose rows and columns are indexed by $y_0r$ and $y_0r'$, respectively.

\begin{lem}\label{lem:order}
Let $(Y,T)$ be a $3$-equivalenced scheme, $y_0 \in Y$ and $u, v, w \in T_3$.
Assume $|w \cap (y_0u \times y_0v)| = 3$.
Then there exist a permutation $\tau$ on $Y$ and a scheme $(Y^\tau, T^\tau)$ isomorphic to $(Y, T)$ such that
$\tau|_{Y \setminus y_0v} = id$ and the adjacency matrix of $(w \cap (y_0u \times y_0v))^\tau$ is the identity matrix,
where $id$ is the identity permutation.
\end{lem}
\begin{proof}
If the adjacency matrix of $w \cap (y_0u \times y_0v)$ is the identity matrix, then we set $\tau = id$.

Assume that the adjacency matrix of $w \cap (y_0u \times y_0v)$ is not the identity matrix.
Then we can choose a permutation $\tau$ on $Y$ such that
$\tau|_{Y \setminus y_0v} = id$ and the adjacency matrix of $(w \cap (y_0u \times y_0v))^\tau$ is the identity matrix.
It follows from the definition that $(Y^\tau, T^\tau)$ is a scheme isomorphic to $(Y, T)$.
\end{proof}

\begin{prop}\label{lem:well-order}
For a given $3$-equivalenced scheme $(Y, T)$,
there exists a scheme $(Y', T')$ isomorphic to $(Y, T)$ such that
the point set of $T'_{y_{0}}$ is well-ordering.
\end{prop}
\begin{proof}
Since $|u^\ast v|= 2$ or $3$ for $u, v \in T^\times$, there exists at least one basic relation $t \in T$
such that $|t \cap (y_0u \times y_0v)|=3$.
We denote all elements of $T^\times$ by $t_1, t_2, \dotsc, t_l$, where $l = \frac{|Y| - 1}{3}$.

First, we take an element $t \in T^\times$ such that $|t \cap (y_0t_1 \times y_0t_2)|=3$.
Replacing $u, v$ and $w$ by $t_1, t_2$ and $t$,
we apply Lemma \ref{lem:order} for them.
Then we can obtain a scheme $(Y^{\tau_1}, T^{\tau_1})$ isomorphic to $(Y, T)$.

Again we denote all elements of $(T^{\tau_1})^\times$ by $t_1^{\tau_1}, t_2^{\tau_1}, \dotsc, t_l^{\tau_1}$.
We take an element $t' \in (T^{\tau_1})^\times$ such that $|t' \cap (y_0t_2^{\tau_1} \times y_0t_3^{\tau_1})|=3$.
Replacing $u, v$ and $w$ by $t_2^{\tau_1}, t_3^{\tau_1}$ and $t'$,
we apply Lemma \ref{lem:order} for them.
Then we can obtain a scheme $(Y^{\tau_1 \tau_2}, T^{\tau_1 \tau_2})$ isomorphic to $(Y^{\tau_1}, T^{\tau_1})$.




\vskip10pt
Repeating the above argument, we can obtain a scheme $(Y', T') \simeq (Y, T)$ with the following property:
For each $1 \leq i \leq l-1$, there exists the only basic relation $t_{m_i}' \in T'$ such that
the adjacency matrix of $t_{m_i}' \cap (y_0t'_i \times y_0t'_{i+1})$ is the identity matrix, 
where $(T')^\times = \{t_1', t_2', \dotsc, t_l' \}$ and $1 \leq m_i \leq l$.

We denote $t_{m_i}' \cap (y_0t'_i \times y_0t'_{i+1})$ by $\overline{t_{m_i}'}$.
Then $T'_{y_{0}}$ contains $\{ \overline{t_{m_i}'} \mid 1 \leq i \leq l-1 \}$.
We can check that $\overline{t_{m_i}'} ~ \overline{t_{m_{i+1}}'} \cdots \overline{t_{m_j}'} \in T'_{y_{0}}$ for $1 \leq i \leq j \leq l-1$.
The adjacency matrix of $\overline{t_{m_i}'} ~ \overline{t_{m_{i+1}}'} \cdots \overline{t_{m_j}'}$ is the identity matrix.
This and Lemma \ref{lem:semiregular} imply that the point set of $T'_{y_{0}}$ is well-ordering.
\end{proof}

\vskip10pt
By Proposition \ref{lem:well-order},
from now on, we assume that the point set of $T_{y_{0}}$ is well-ordering.

\subsection{Central primitive idempotents of $\mathcal{T}(X\times Y, S \wr T, (x_0, y_0))$}

Set $F^{(t)} = (X,y_{0}t)$ and
$U^{(t)}=(S\wr T)_{F^{(t)}}$ for $t\in T$.

\vskip15pt
For $\chi \in \mathrm{Irr}(\mathcal{T}(U^{(1_Y)})) \setminus \{ 1_{\mathcal{T}(U^{(1_Y)})} \}$, define
$$\tilde{e}_{\chi} = e_{\chi}\otimes \varepsilon_{\{y_{0}\}}\in \mathcal{T}(S\wr T).$$

For $t \in T_{3}$ and
$\psi \in \mathrm{Irr}(\mathcal{A}(S)) \setminus \{1_{\mathcal{A}(S)}\}$, define
$$\hat{e}_{\psi} = e_{\psi}\otimes \varepsilon_{y_{0}t} \in \mathcal{T}(S\wr T).$$

It is easy to see that they are idempotents of $\mathcal{T}(S\wr T)$.

\begin{lem}[\cite{hkm}, Lemma 4.2 and 4.4]\label{lem:32}
For $\chi \in \mathrm{Irr}(\mathcal{T}(U^{(1_Y)})) \setminus \{ 1_{\mathcal{T}(U^{(1_Y)})} \}$,
$\tilde{e}_{\chi}$ is a central primitive idempotent of $\mathcal{T}(S \wr T)$.
\end{lem}

\begin{lem}\label{lem:34}
For $t \in T_{3}$ and $\psi \in \mathrm{Irr}(\mathcal{A}(S)) \setminus \{1_{\mathcal{A}(S)}\}$,
$\hat{e}_{\psi}$ is a central primitive idempotent of $\mathcal{T}(S \wr T)$.
\end{lem}
\begin{proof}
First, we show that $\hat{e}_{\psi}$ commutes with $\sigma_{s}\otimes I_{Y}$ ($s \in S$), $J_{X}\otimes \sigma_{u}$ ($u \in
T\setminus\{1_Y\}$), $\varepsilon_{x_{0}s}\otimes \varepsilon_{\{y_{0}\}}$ ($s \in S$), and $I_{X}\otimes \varepsilon_{y_{0}u}$ ($u\in T\setminus\{1_Y\}$).

For $s \in S$, we have
\[\hat{e}_{\psi} (\sigma_{s}\otimes I_{Y}) = \sum_{u\in T} \hat{e}_{\psi}(\sigma_{s}\otimes \varepsilon_{y_0 u}) = (e_{\psi}\otimes \varepsilon_{y_0 t} )(\sigma_{s}\otimes \varepsilon_{y_0 t}).\]
Since $e_{\psi}$ commutes with $\sigma_s$, we have
\[\hat{e}_{\psi} (\sigma_{s}\otimes I_{Y}) = (\sigma_{s}\otimes I_{Y}) \hat{e}_{\psi}.\]

Since $t \neq 1_Y$, we have
\[\hat{e}_{\psi}(\varepsilon_{x_{0}s}\otimes \varepsilon_{\{y_{0}\}}) = (\varepsilon_{x_{0}s}\otimes \varepsilon_{\{y_{0}\}})\hat{e}_{\psi} = 0.\]

Since $e_{1_{\mathcal{A}(S)}}=|X|^{-1}J_X$ and $e_{1_{\mathcal{A}(S)}} e_\psi = e_\psi e_{1_{\mathcal{A}(S)}}=0$, we have
\[\hat{e}_{\psi}(J_X\otimes \sigma_u) = (J_X\otimes \sigma_u)\hat{e}_{\psi} = 0.\]
Also, $\hat{e}_{\psi}(I_{X}\otimes \varepsilon_{y_{0}u}) = (I_{X}\otimes \varepsilon_{y_{0}u})\hat{e}_{\psi}$ is trivial.

\vskip5pt
Next, we show that $\hat{e}_{\psi}$ is primitive.
The map  $\pi:\mathcal{T}(S\wr T) \rightarrow \hat{e}_{\psi}\mathcal{T}(S\wr T)$ is a projection.
Actually, $\hat{e}_{\psi}\mathcal{T}(S\wr T)$ is naturally isomorphic to $e_{\psi}\mathcal{A}(S)$.
Since $e_{\psi}$ is a central primitive idempotent of $\mathcal{A}(S)$, $\hat{e}_{\psi}$ is primitive.
\end{proof}

\vskip15pt
Now we define the other central primitive idempotents of $\mathcal{T}(S \wr T)$.

\vskip15pt
Put
\begin{enumerate}
\item $\omega$ is a cube root of unity,
\item $y_0 t = \{ y_{t(1)}, y_{t(2)}, y_{t(3)} \}$ and
\item $y_0 t' = \{ y_{t'(1)}, y_{t'(2)}, y_{t'(3)} \}$
\end{enumerate}
such that the adjacency matrix of $\{ (y_{t(1)}, y_{t'(1)}), (y_{t(2)}, y_{t'(2)}), (y_{t(3)}, y_{t'(3)}) \}$  is the identity matrix.

\vskip15pt
For all $t, t' \in T_3$, define
\begin{eqnarray*}
G_{y_0 t, y_0 t'} &=& \frac{1}{3|X|}J_{X}\otimes (\varepsilon_{y_{0}t, y_{0}t'} + \omega \overline{\varepsilon_{y_{0}t, y_{0}t'}} + \omega^2 \underline{\varepsilon_{y_{0}t, y_{0}t'}})
\end{eqnarray*}
and
\begin{eqnarray*}
G'_{y_0 t, y_0 t'} &=& \frac{1}{3|X|}J_{X}\otimes (\varepsilon_{y_{0}t, y_{0}t'} + \omega^2 \overline{\varepsilon_{y_{0}t, y_{0}t'}} + \omega \underline{\varepsilon_{y_{0}t, y_{0}t'}}),
\end{eqnarray*}
where \[\varepsilon_{y_{0}t, y_{0}t'} := J_{\{y_{t(1)}\}, \{y_{t'(1)}\}} + J_{\{y_{t(2)}\}, \{y_{t'(2)}\}} + J_{\{y_{t(3)}\}, \{y_{t'(3)}\}},\]
\[\overline{\varepsilon_{y_{0}t, y_{0}t'}} := J_{\{y_{t(1)}\}, \{y_{t'(2)}\}} + J_{\{y_{t(2)}\}, \{y_{t'(3)}\}} +  J_{\{y_{t(3)}\}, \{y_{t'(1)}\}},\]
and \[\underline{\varepsilon_{y_{0}t, y_{0}t'}} := J_{\{y_{t(1)}\}, \{y_{t'(3)}\}} + J_{\{y_{t(2)}\}, \{y_{t'(1)}\}} + J_{\{y_{t(3)}\}, \{y_{t'(2)}\}}.\]

It is easy to see that $\{G_{y_0 t, y_0 t'} \mid t, t' \in T_3 \}$ and $\{G'_{y_0 t, y_0 t'} \mid t, t' \in T_3 \}$
are linearly independent subsets of $\mathcal{T}(S \wr T)$.

\begin{lem}\label{lem:ideal}
$\langle\{G_{y_0 t, y_0 t'} \mid t, t' \in T_3 \}\rangle$ is an ideal
and isomorphic to $Mat_{T_3}(\mathbb{C})$.
\end{lem}
\begin{proof}
First, we prove that $\sigma_u G_{y_0 t, y_0 t'}, G_{y_0 t, y_0 t'} \sigma_u \in \langle\{G_{y_0 t, y_0 t'} \mid t, t' \in T_3 \}\rangle$
for $u \in S \wr T$.

If $\sigma_u G_{y_0 t, y_0 t'} \neq 0$, then $(\sigma_u)_{X \times y_0 h, X \times y_0t} \neq 0$ for some $h \in T_3$.
This means that $\sigma_u = J_{X}\otimes \sigma_{t_1}$ for some $t_1 \in T_3$ and $t_1 \in h^*t$.
By Lemma \ref{lem:inter1}, $|h^\ast t| = 2$ or $3$.
This means that $|t_1 \cap (y_0h \times y_0t)| = 3$ or $6$.
Note that the adjacency matrix of $t_1 \cap (y_0h \times y_0t)$ is a sum of at most two matrices given in (\ref{A}).

So, $\sigma_u G_{y_0 t, y_0 t'}$ is either $(J_{X}\otimes \varepsilon_1)G_{y_0 t, y_0 t'}$ or
$(J_{X}\otimes (\varepsilon_1 + \varepsilon_2))G_{y_0 t, y_0 t'}$,
where $\varepsilon_1, \varepsilon_2 \in \{ \varepsilon_{y_{0}h, y_{0}t}, \overline{\varepsilon_{y_{0}h, y_{0}t}}, \underline{\varepsilon_{y_{0}h, y_{0}t}} \}$ are distinct.

By calculation, we have the followings:
\begin{eqnarray*}
(J_{X}\otimes \varepsilon_1)G_{y_0 t, y_0 t'} &=& (J_{X}\otimes \varepsilon_1) (\frac{1}{3|X|}J_{X}\otimes (\varepsilon_{y_{0}t, y_{0}t'} + \omega \overline{\varepsilon_{y_{0}t, y_{0}t'}} + \omega^2 \underline{\varepsilon_{y_{0}t, y_{0}t'}}))  \\
                                  &=& \frac{1}{3}J_{X}\otimes \omega^i(\varepsilon_{y_{0}h, y_{0}t'} + \omega \overline{\varepsilon_{y_{0}h, y_{0}t'}} + \omega^2 \underline{\varepsilon_{y_{0}h, y_{0}t'}})
\end{eqnarray*}
for some $0 \leq i \leq 2$;

\begin{eqnarray*}
(J_{X}\otimes (\varepsilon_1 + \varepsilon_2))G_{y_0 t, y_0 t'} &=& (J_{X}\otimes (\varepsilon_1 + \varepsilon_2)) (\frac{1}{3|X|}J_{X}\otimes (\varepsilon_{y_{0}t, y_{0}t'} + \omega \overline{\varepsilon_{y_{0}t, y_{0}t'}} + \omega^2 \underline{\varepsilon_{y_{0}t, y_{0}t'}}))  \\
                                &=& \frac{1}{3}J_{X}\otimes (\omega^i + \omega^j)(\varepsilon_{y_{0}h, y_{0}t'} + \omega \overline{\varepsilon_{y_{0}h, y_{0}t'}} + \omega^2 \underline{\varepsilon_{y_{0}h, y_{0}t'}})
\end{eqnarray*}
for some distinct $0 \leq i, j \leq 2$.

Thus, $\sigma_u G_{y_0 t, y_0 t'}$ is either $|X|\omega^iG_{y_0 h, y_0 t'}$ or $|X|(\omega^i + \omega^j)G_{y_0 h, y_0 t'}$
for some distinct $0 \leq i, j \leq 2$.
Whichever the case may be, it implies $\sigma_u G_{y_0 t, y_0 t'} \in \langle\{G_{y_0 t, y_0 t'} \mid t, t' \in T_3 \}\rangle$.
Similarly, we can show $G_{y_0 t, y_0 t'} \sigma_u  \in \langle\{G_{y_0 t, y_0 t'} \mid t, t' \in T_3 \}\rangle$.

\vskip15pt
For $u \in S \wr T$ and $t, t' \in T_3$,
clearly
\[\varepsilon_{(x_0,y_0)u} G_{y_0 t, y_0 t'} = \delta_{(x_0,y_0)u (X\times y_0 t)}G_{y_0 t, y_0 t'} \in \langle\{G_{y_0 t, y_0 t'} \mid t, t' \in T_3 \}\rangle\]
and
\[G_{y_0 t, y_0 t'}\varepsilon_{(x_0,y_0)u} \in \langle\{G_{y_0 t, y_0 t'} \mid t, t' \in T_3 \}\rangle.\]

Thus, $\langle\{G_{y_0 t, y_0 t'} \mid t, t' \in T_3 \}\rangle$ is an ideal.

\vskip15pt
Next, we prove that $G_{y_0 t, y_0 t'} G_{y_0 t''', y_0 t''} = \delta_{t't'''}G_{y_0 t, y_0 t''}$.
It suffices to show that $G_{y_0 t, y_0 t'}G_{y_0 t', y_0 t''} = G_{y_0 t, y_0 t''}$.
By calculation, we have
\begin{eqnarray*}
G_{y_0 t, y_0 t'}G_{y_0 t', y_0 t''} =
(\frac{1}{3|X|}J_{X}\otimes (\varepsilon_{y_{0}t, y_{0}t'} + \omega \overline{\varepsilon_{y_{0}t, y_{0}t'}} + \omega^2 \underline{\varepsilon_{y_{0}t, y_{0}t'}})) \\
(\frac{1}{3|X|}J_{X}\otimes (\varepsilon_{y_{0}t', y_{0}t''} + \omega \overline{\varepsilon_{y_{0}t', y_{0}t''}} + \omega^2 \underline{\varepsilon_{y_{0}t', y_{0}t''}})) \\
= \frac{1}{9|X|}J_{X}\otimes 3(\varepsilon_{y_{0}t, y_{0}t''} + \omega \overline{\varepsilon_{y_{0}t, y_{0}t''}} + \omega^2 \underline{\varepsilon_{y_{0}t, y_{0}t''}}) = G_{y_0 t, y_0 t''}.
\end{eqnarray*}

\vskip15pt
Finally, we prove that $\langle\{G_{y_0 t, y_0 t'} \mid t, t' \in T_3 \}\rangle \cong Mat_{T_3}(\mathbb{C})$.
For $t, t' \in T_3$, let $e_{tt'}$ be a $|T_3| \times |T_3|$ matrix whose $(t, t')$-entry is $1$ and whose other entries are all zero.
Then the linear map $\varphi: \langle\{G_{y_0 t, y_0 t'} \mid t, t' \in T_3 \}\rangle \rightarrow Mat_{T_3}(\mathbb{C})$ defined by
$\varphi(G_{y_0 t, y_0 t'}) = e_{tt'}$ is an isomorphism.
\end{proof}

Similarly to the proof of Lemma \ref{lem:ideal}, we can prove the following lemma.
\begin{lem}\label{lem:ideal2}
$\langle\{G'_{y_0 t, y_0 t'} \mid t, t' \in T_3 \}\rangle$ is an ideal
and isomorphic to $Mat_{T_3}(\mathbb{C})$.
\end{lem}

\vskip15pt
Define \[e_{\eta_1} = \sum_{t \in T_3} \frac{1}{3|X|}J_{X}\otimes (\varepsilon_{y_{0}t} + \omega \overline{\varepsilon_{y_{0}t}} + \omega^2 \underline{\varepsilon_{y_{0}t}})\]
and
\[e_{\eta_2} = \sum_{t \in T_3} \frac{1}{3|X|}J_{X}\otimes (\varepsilon_{y_{0}t} + \omega^2 \overline{\varepsilon_{y_{0}t}} + \omega \underline{\varepsilon_{y_{0}t}}),\]
where
\[\overline{\varepsilon_{y_{0}t}} := J_{\{y_{t(1)}\}, \{y_{t(2)}\}} + J_{\{y_{t(2)}\}, \{y_{t(3)}\}} +  J_{\{y_{t(3)}\}, \{y_{t(1)}\}}\]
and \[\underline{\varepsilon_{y_{0}t}} := J_{\{y_{t(1)}\}, \{y_{t(3)}\}} + J_{\{y_{t(2)}\}, \{y_{t(1)}\}} + J_{\{y_{t(3)}\}, \{y_{t(2)}\}}.\]

Then, by Lemma \ref{lem:ideal} and \ref{lem:ideal2}, $e_{\eta_1}$ and $e_{\eta_2}$ are central primitive idempotents of $\mathcal{T}(S\wr T)$.

\begin{lem}\label{lem:sum}
The sum of $e_{1_{\mathcal{T}(S\wr T)}}$, $\tilde{e}_{\chi}$'s, $\hat{e}_{\psi}$'s, $e_{\eta_1}$ and $e_{\eta_2}$ is the identity element.
\end{lem}
\begin{proof}
It is easy to see that
\begin{eqnarray*}
e_{1_{\mathcal{T}(S\wr T)}} &=& e_{1_{\mathcal{T}(U^{(1_Y)})}}\otimes \varepsilon_{\{y_{0}\}}
  + \sum_{t\in T_3}\frac{1}{3|X|}\varepsilon_{F^{(t)}}J_{X\times Y}\varepsilon_{F^{(t)}},
\end{eqnarray*}

\begin{eqnarray*}
\sum_{\chi\in \mathrm{Irr}(\mathcal{T}(U^{(1_Y)}))^\times}\tilde{e}_\chi
  = \varepsilon_{F^{(1_Y)}}I_{X\times Y}\varepsilon_{F^{(1_Y)}}
  -e_{1_{\mathcal{T}(U^{(1_Y)})}}\otimes \varepsilon_{\{y_0\}},
\end{eqnarray*}

\begin{eqnarray*}
\sum_{\psi \in \mathrm{Irr}(\mathcal{A}(S))^\times}\hat{e}_\psi
  = \varepsilon_{F^{(t)}}I_{X\times Y}\varepsilon_{F^{(t)}}
  -\frac{1}{|X|}(J_{X\times y_{t(1)}} + J_{X\times y_{t(2)}} + J_{X\times y_{t(3)}})
\end{eqnarray*}
for each $t \in T_3$,

\begin{eqnarray*}
e_{\eta_1} = \sum_{t \in T_3} \frac{1}{3|X|}J_{X}\otimes (\varepsilon_{y_{0}t} + \omega \overline{\varepsilon_{y_{0}t}} + \omega^2 \underline{\varepsilon_{y_{0}t}})
\end{eqnarray*}
and
\begin{eqnarray*}
e_{\eta_2} = \sum_{t \in T_3} \frac{1}{3|X|}J_{X}\otimes (\varepsilon_{y_{0}t} + \omega^2 \overline{\varepsilon_{y_{0}t}} + \omega \underline{\varepsilon_{y_{0}t}}).
\end{eqnarray*}

Thus, we have
\begin{eqnarray*}
e_{1_{\mathcal{T}(S\wr T)}}
+ \sum_{\chi\in \mathrm{Irr}(\mathcal{T}(U^{(1_Y)}))^\times}\tilde{e}_\chi
+ \sum_{t \in T_3} \sum_{\psi \in \mathrm{Irr}(\mathcal{A}(S))^\times}\hat{e}_\psi
+ e_{\eta_1} + e_{\eta_2} = I_{X\times Y}.
\end{eqnarray*}
\end{proof}

\section{Main result}\label{sec:mainresult}
In conclusion,  we have determined all central primitive idempotents of Terwilliger algebras of wreath products by 3-equivalenced schemes.
Combining Section \ref{sec:main} and \cite[Theorem 4.1]{hkm} gives the following theorem.

\begin{thm}\label{thm:maintheorem}
Let $(X,S)$ be a scheme and $(Y,T)$ a 3-equivalenced scheme.
Fix $x_0 \in X$ and $y_0 \in Y$, and consider the wreath product $(X \times Y, S \wr T)$. Then
\begin{enumerate}
\item If $|T| = 2$, then
\begin{eqnarray*}
\{ e_{1_{\mathcal{T}(S\wr T)}} \} & \cup & \{ \tilde{e}_\chi \mid \chi \in \mathrm{Irr}(\mathcal{T}(U^{(1_Y)}))^\times  \}  \\
                                  & \cup & \bigcup_{t \in T_3}  \{ \hat{e}_\psi \mid \psi \in \mathrm{Irr}(\mathcal{A}(S))^\times \}
\end{eqnarray*}
is the set of all central primitive idempotents of $\mathcal{T}(X\times Y, S \wr T, (x_0, y_0))$.
\item If $|T| > 2$, then
\begin{eqnarray*}
\{ e_{1_{\mathcal{T}(S\wr T)}} \} & \cup & \{ \tilde{e}_\chi \mid \chi \in \mathrm{Irr}(\mathcal{T}(U^{(1_Y)}))^\times  \}  \\
\{ e_{\eta_1}, e_{\eta_2} \}                    & \cup & \bigcup_{t \in T_3}  \{ \hat{e}_\psi \mid \psi \in \mathrm{Irr}(\mathcal{A}(S))^\times \}
\end{eqnarray*}
is the set of all central primitive idempotents of $\mathcal{T}(X\times Y, S \wr T, (x_0, y_0))$.
\end{enumerate}
\end{thm}

\bibstyle{plain}

\end{document}